\documentclass[11pt]{amsart}

\usepackage{graphicx}
\usepackage{mathptmx}

\usepackage{amscd}
\usepackage{amsthm}
\usepackage{amsxtra}
\usepackage{a4wide}
\usepackage{latexsym}
\usepackage{amssymb}
\usepackage{amsfonts}
\usepackage{amsmath}
\usepackage{amsrefs}
\usepackage{mathrsfs}

\usepackage{upref}
\usepackage{txfonts}
\usepackage{graphicx}

\usepackage[bookmarksnumbered, colorlinks, plainpages]{hyperref}
\hypersetup{colorlinks=true,linkcolor=red, anchorcolor=green, citecolor=cyan, urlcolor=red, filecolor=magenta, pdftoolbar=true}

\usepackage[left=2cm, right=2cm, top=2cm, bottom=2cm]{geometry}

\allowdisplaybreaks

\theoremstyle{plain}
\newtheorem{thm}{Theorem}[section]
\newtheorem{lem}[thm]{Lemma}

\theoremstyle{definition}
\newtheorem{rem}[thm]{Remark}

\newtheorem{defi}[thm]{Definition}

\newtheorem{conv}[thm]{Convention}

\numberwithin{thm}{section}
\numberwithin{equation}{section}

\newcommand\ddfrac[2]{\frac{\displaystyle #1}{\displaystyle #2}}
\newcommand*{\medcap}{\mathbin{\scalebox{1.7}{\ensuremath{\cap}}}}

\def\loc{\operatorname{loc}}

\def\supp{\operatorname{supp}}

\def\esup{\operatornamewithlimits{ess\,sup}}

\def\Id{\operatorname{I}}

\def\ap{\approx}
\def\qq{\qquad}
\def\rw{\rightarrow}

\def\hra{\hookrightarrow}

\def\M{\mathcal M}

\def\o{\omega}

\def\la{\lambda}

\def\vp{\varphi}

\def\i{\infty}
\def\I{(0,\i)}
\def\t{\theta}

\def\R{\mathbb R}

\def\R{\mathbb R}

\def\M{\mathfrak M}

\def\W{{\mathcal W}}

\def\mp{{\mathfrak M}}

\def\rn{\R^n}

\def\o{\omega}
\def\O{\Omega}

\def\la{\lambda}

\def\vp{\varphi}

\def\i{\infty}
\def\I{(0,\i)}
\def\t{\theta}

\def\dual{\,^{^{\bf c}}\!}

\def\Btd {\dual\Bt}
\def\Brd {\dual\Br}

\def\Bt {{B(0,t)}}

\def\Br {{B(0,r)}}
\def\Bxr {{B(x,r)}}

\def\Lploc{L_p^{\rm loc}(\rn)}

\def\LM {LM_{p\t,\o}}

\def\LMd {\dual{\LM}}

\def\Lploc{L_p^{\rm loc}(\rn)}

\def\LM {LM_{p\t,\o}}
\def\LMv {LM_{p\t,\o}(\rn,v)}
\def\LMd {\dual{\LM}}
\def\LMvd {\dual{\LMv}}
\def\Ot {\Omega_{\t}}
\def\Otd {\dual{\Ot}}

\begin{document}

\title[Embeddings between $\LMvd$ and $\LMv$]{Embeddings between weighted complementary local Morrey-type spaces and weighted local Morrey-type spaces}

\author[A. Gogatishvili]{Amiran Gogatishvili}
\address{Institute of Mathematics \\
    Academy of Sciences of the Czech Republic \\
    \v Zitn\'a~25 \\
    115~67 Praha~1, Czech Republic} \email{gogatish@math.cas.cz}

\author[R.Ch.Mustafayev]{Rza Mustafayev}
\address{Department of Mathematics \\ Faculty of Science and Arts \\ Kirikkale
    University \\ 71450 Yahsihan, Kirikkale, Turkey}
\email{rzamustafayev@gmail.com}

\author[T.~{\"U}nver]{Tu\v{g}\c{c}e {\"U}nver}
\address{Department of Mathematics \\ Faculty of Science and Arts \\ Kirikkale
    University \\ 71450 Yahsihan, Kirikkale, Turkey}
\email{tugceunver@gmail.com}

\thanks{}

\subjclass[2010]{Primary 46E30; Secondary 26D10.}

\keywords{local Morrey-type spaces, embeddings, iterated Hardy inequalities}

\begin{abstract}
	In this paper embeddings between  weighted complementary local Morrey-type spaces ${\,^{^{\bf c}}\!}LM_{p\theta,\omega}({\mathbb R}^n,v)$ and  weighted local Morrey-type spaces $LM_{p\theta,\omega}({\mathbb R}^n,v)$ are characterized. In particular, two-sided estimates of the optimal constant $c$ in the inequality
	\begin{equation*}
	\bigg( \int_0^{\infty} \bigg( \int_{B(0,t)} f(x)^{p_2}v_2(x)\,dx \bigg)^{\frac{q_2}{p_2}} u_2(t)\,dt\bigg)^{\frac{1}{q_2}} \le c \bigg( \int_0^{\infty} \bigg( \int_{{\,^{^{\bf c}}\!}B(0,t)} f(x)^{p_1} v_1(x)\,dx\bigg)^{\frac{q_1}{p_1}} u_1(t)\,dt\bigg)^{\frac{1}{q_1}}
	\end{equation*}
	are obtained, where $p_1,\,p_2,\,q_1,\,q_2 \in (0,\infty)$, $p_2 \le q_2$ and $u_1,\,u_2$ and $v_1,\,v_2$ are weights on $(0,\infty)$ and ${\mathbb R}^n$, respectively. The proof is based on the combination of duality techniques with estimates of optimal constants of the embeddings between	weighted local Morrey-type and complementary local Morrey-type spaces and weighted Lebesgue spaces, which reduce the problem to the solutions of the iterated Hardy-type inequalities.
\end{abstract}

\maketitle


\section{Introduction}\label{introduction}

The classical Morrey spaces $\mathcal{M}_{p, \lambda} \equiv \mathcal{M}_{p, \lambda} (\rn)$,
were  introduced by C.~Morrey in \cite{M1938} in order to study regularity questions which appear in the Calculus of Variations,
and defined as follows:  for $0 \le \lambda \le n$ and $1\le p \le \infty$,
$$
\mathcal{M}_{p,\lambda} : = \left\{ f \in   \Lploc:\,\left\| f\right\|_{\mathcal{M}_{p,\lambda }} : =
\sup_{x\in \rn, \; r>0 }
r^{\frac{\lambda-n}{p}} \|f\|_{L_{p}(B(x,r))} <\infty\right\},
$$
where $\Bxr$ is the open ball centered at $x$ of radius $r$.

Note that $\mathcal{M}_{p,0}(\rn) = L_{\infty}(\rn)$ and ${\mathcal M}_{p,n}(\rn) = L_{p}(\rn)$.

These spaces describe local regularity more precisely than Lebesgue spaces and appeared to be quite useful in the study of the local
behavior of solutions to partial differential equations, a priori estimates and other topics in PDE (cf. \cite{giltrud}).    

The classical Morrey spaces were widely investigated during the last decades, including the study of classical operators of Harmonic and Real Analysis - maximal, singular and potential operators - in generalizations of these spaces (so-called Morrey-type spaces). The local Morrey-type spaces and the complementary local Morrey-type spaces introduced by Guliyev in his doctoral thesis \cite{GulDoc}. 

The local Morrey-type spaces $\LM$ and the complementary local Morrey-type spaces $\dual \LM$ were intensively studied during the last decades. The research mainly includes the study of the boundedness of classical operators in these spaces (see, for instance, \cite{Bur1,Bur2,BurGol,BurHus1,BurGulHus2,BurGulHus1,BurGulSerTar,BGGM,BGGM1}), and the
investigation of the functional-analytic properties of them and relation of these spaces with other known
function spaces (see, for instance, \cites{BurNur,batsaw,mu_emb}). We refer the reader to the surveys \cite{Bur1}
and \cite{Bur2} for a comprehensive discussion of the history of $\LM$ and $\dual \LM$. 

Let $A$ be any measurable subset of $\rn$, $n \ge 1$. By $\mp (A)$
we denote the set of all measurable functions on $A$. The symbol
$\mp^+ (A)$ stands for the collection of all $f\in\mp (A)$ which
are non-negative on $A$. The family of all weight functions (also
called just weights) on $A$, that is, measurable, positive and
finite a.e. on $A$, is given by $\W (A)$.

For $p\in (0,\i]$, we define the functional
$\|\cdot\|_{p,A}$ on $\mp (A)$ by
\begin{equation*}
\|f\|_{p,A} : = \bigg\{\begin{array}{cl}
\left(\int_A |f(x)|^p \,dx \right)^{1/p} & \qq\mbox{if}\qq p<\i \\
\esup_{A} |f(x)| & \qq\mbox{if}\qq p=\i
\end{array}
\bigg..
\end{equation*}

If $w\in \W(A)$, then the weighted Lebesgue space
$L_p(w,A)$ is given by
\begin{equation*}
L_{p}(w,A) \equiv L_{p,w}(A) : = \{f\in \mp (A):\,\, \|f\|_{p,w,A} := \|fw\|_{p,A} < \infty \}.
\end{equation*}
When $A=\rn$, we often write simply $L_{p,w}$ and $L_p(w)$ instead of $L_{p,w}(A)$ and $L_p(w,A)$, respectively. 

Throughout the paper, we always denote by $c$ and $C$ a positive
constant, which is independent of main parameters but it may vary
from line to line. However a constant with subscript
such as $c_1$ does not change in different occurrences. By
$a\lesssim b$, ($b\gtrsim a$) we mean that $a\leq \la b$, where
$\la>0$ depends on inessential parameters. If $a\lesssim b$ and
$b\lesssim a$, we write $a\approx b$ and say that $a$ and $b$ are
equivalent.  We will denote by $\bf 1$ the function ${\bf 1}(x) =
1$, $x \in \R$.

Given two quasi-normed vector spaces $X$ and $Y$, we write $X=Y$ if
$X$ and $Y$ are equal in the algebraic and the topological sense
(their quasi-norms are equivalent). The symbol $X\hookrightarrow Y$
($Y \hookleftarrow X$) means that $X\subset Y$ and the natural
embedding $\Id$ of $X$ in $Y$ is continuous, that is, there exist a
constant $c > 0$ such that $\|z\|_Y \le c\|z\|_X$ for all $z\in X$.
The best constant of the embedding $X\hookrightarrow Y$ is
$\|\Id\|_{X \rw Y}$.

The weighted local Morrey-type spaces $\LMv$ and weighted complementary
local Morrey-type spaces $\LMv$ are defined as follows: Let  $0 <p, \theta \le \infty$. 
Assume that $\o \in \mp^+ \I$ and $v \in\W(\rn)$.      
$$
\LMv : = \bigg\{f\in L_{p,v}^{\loc}(\rn):\, \| f \|_{\LMv} < \infty\bigg\}, \\
$$
where
$$
\| f \|_{\LMv} : = \big\| \|f\|_{p,v,\Br} \big\|_{\t,\o,(0,\infty)},
$$
and
$$
\LMvd :  = \bigg\{f\in \medcap_{t>0} L_{p,v}(\dual B(0,t)):\,\| f\|_{\LMvd}  < \infty\bigg\},
$$
where
$$
\| f\|_{\LMvd} : = \big\| \|f\|_{p,v,\Brd} \big\|_{\t,\o, (0,\infty)}.
$$
\begin{rem}
	In \cite{BurHus1} and \cite{BurGulHus1} it were proved that the
	spaces $\LM (\rn) : = \LM (\rn, \bf 1)$ and $\LMd (\rn) : = \LMd
	(\rn, \bf 1)$ are non-trivial, i.e. consists not only of functions
	equivalent to $0$ on $\rn$, if and only if
	\begin{equation}\label{11111}
	\|\o\|_{\t,(t,\i)} < \i, \qq \mbox{for some} \qq t >0,
	\end{equation}
	and
	\begin{equation}\label{1111111}
	\|\o\|_{\t,(0,t)} < \i,\qq \mbox{for some} \qq t >0,
	\end{equation}
	respectively. The same conclusion is true for $\LMv$ and $\LMvd$ for
	any $v \in \W(\rn)$.
\end{rem}

The proof of the following statement is straightforward.
\begin{lem} \label{shrinkagelemma}
	\rm{(i)} If $\|\o\|_{\t,(t_1,\infty)} = \infty$ for some $t_1 > 0$, then
	$$
	f\in \LMv \Rightarrow f = 0 \quad \mbox{a.e. on} \quad B(0,t_1).
	$$
	
	\rm{(ii)} If $\|\o\|_{\t,(0,t_2)} = \infty$ for some $t_2 > 0$, then
	$$
	f\in \LMvd \Rightarrow f = 0 \quad \mbox{a.e. on} \quad \dual B(0,t_2).
	$$
\end{lem}

Let $0< \t \le \i$. We denote by
\begin{align*}
\Ot : & = \big\{\o \in \mp^+ \I:\, 0<\|\o\|_{\t,(t,\i)}<\i,~ t>0 \big\}, \\
\Otd : & = \big\{\o \in \mp^+ \I:\,0<\|\o\|_{\t,(0,t)}<\i,~ t>0 \big\}.	
\end{align*}

Let $v \in \W(\rn)$. It is easy to see that $\LMv$ and $\LMvd$ are
quasi-normed vector spaces when $\o \in \Ot$ and $\o \in \Otd$,
respectively.

The following statements are immediate consequences of Fubini's Theorem and were observed in \cite{BurHus1} and \cite{BurGulHus1}, for $v = 1$, respectively. 
\begin{lem}\label{LMpp}
	Let $0 < p \leq \infty$ and $v\in \W(\rn)$. Then 
	
	{\rm (i)} $LM_{pp,\o}(\rn,v)=L_p(w)$, where	$w(x) := v(x)\|\o\|_{p,(|x|,\infty)}, ~ x \in \rn$.
	
	{\rm (ii)} $\dual LM_{pp,\o}(\rn,v)=L_p(w)$, where	$w(x) := v(x)\|\o\|_{p,(0,|x|)}, ~ x \in \rn$.
\end{lem}

Recall that the embeddings between weighted local Morrey-type spaces and weighted Lebesgue spaces, that
is, the embeddings 
\begin{align}
L_{p_1}(v_1) & \hra LM_{p_2 \t, \o}(\rn,v_2), \label{emb1}\\
L_{p_1}(v_1) & \hra \dual{LM}_{p_2 \t, \o}(\rn,v_2), \label{emb2}\\
L_{p_1}(v_1) & \hookleftarrow LM_{p_2\t, \o}(\rn,v_2), \label{emb3}\\
L_{p_1}(v_1) & \hookleftarrow \dual{LM}_{p_2 \t, \o}(\rn,v_2)
\label{emb4}
\end{align}
are completely characterized in \cite{mu_emb}.

Our principal goal in this paper is to investigate the embeddings
between weighted complementary local Morrey-type spaces and weighted local Morrey type spaces and vice versa, that is, the embeddings
\begin{align}
\dual LM_{p_1 \t_1, \o_1}(\rn,v_1) & \hra LM_{p_2 \t_2, \o_2}(\rn,v_2), \label{mainemb1}\\
LM_{p_1 \t_1, \o_1}(\rn,v_1) & \hra \dual LM_{p_2 \t_2, \o_2}(\rn,v_2).
\label{mainemb2}
\end{align}
An approach used in this paper consist of a duality argument combined with
estimates of optimal constants of embeddings \eqref{emb1} - \eqref{emb4},
which reduce the problem to the solutions of the iterated Hardy-type
inequalities  
\begin{equation}\label{eq.4.11}
\big\| \|H^* f\|_{p,u,(0,\cdot)}\big\|_{q,w,\I}\leq c
\,\|f\|_{\theta,v,\I},~f \in \M^+\I,
\end{equation}
with
$$
(H^*f)(t):= \int_t^{\infty} f(\tau) \, d\tau, \quad t > 0,
$$
where $u,\,v,\,w$ are weights on $(0,\infty)$ and $0 < p,\,q \le \infty$, $1 < \theta < \infty$. There exists different solutions of these inequalities. We will use  characterizations from  \cite{gmp} and \cite{gop}. 

Note that in view of Lemma \ref{LMpp}, embeddings \eqref{mainemb1} - \eqref{mainemb2} contain embeddings \eqref{emb1} - \eqref{emb4} as a special case. Moreover, by the change of
variables $x = {y} / {|y|^2}$ and $t={1} / {\tau}$, it is easy to see that \eqref{mainemb2} is equivalent to the embedding
$$
\dual LM_{p_1\t_1,\tilde \o_1}(\rn,\tilde{v}_1) \hra
LM_{p_2\t_2,\tilde \o_2}(\rn,\tilde{v}_2),
$$
where $\tilde{v}_i (y) = v_i(y/|y|^2)|y|^{-2n/p_i}$ and
$\tilde{\o}_i (\tau) = \tau^{- {2} / {\t_i}} \o_i\big({1} / {\tau}\big)$,
$i=1,2$. This note allows us to concentrate our attention
on characterization of \eqref{mainemb1}. On the negative side of
things we have to admit that the duality approach works only in the
case when, in \eqref{mainemb1} - \eqref{mainemb2}, one has $p_2 \le
\theta_2$. Unfortunately, in the case when $p_2 > \theta_2$ the
characterization of these embeddings remains open.

In particular, we obtain two-sided estimates of the optimal constant
$c$ in the inequality
\begin{equation*}
\bigg( \int_0^{\infty} \bigg( \int_{\Bt} f(x)^{p_2}v_2(x)\,dx \bigg)^{\frac{q_2}{p_2}} u_2(t)\,dt\bigg)^{\frac{1}{q_2}} \le c \bigg( \int_0^{\infty} \bigg( \int_{\Btd} f(x)^{p_1} v_1(x)\,dx\bigg)^{\frac{q_1}{p_1}} u_1(t)\,dt\bigg)^{\frac{1}{q_1}},
\end{equation*}
where $p_1,\,p_2,\,q_1,\,q_2 \in (0,\infty)$, $p_2 \le q_2$ and $u_1,\,u_2$ and $v_1,\,v_2$ are weights on $\I$ and $\rn$, respectively.

The paper is organized as follows. We start with formulations of our main results  in Section~\ref{main.res.}. The proofs of the main results are presented in Section \ref{s7}.


\section{Statement of the main results}\label{main.res.}

We adopt the following usual conventions.
\begin{conv}\label{Notat.and.prelim.conv.1.1}
	{\rm (i)} Throughout the paper we put $0/0 = 0$, $0 \cdot (\pm \i) =
	0$ and $1 / (\pm\i) =0$.
	
	{\rm (ii)} We put
	$$
	p' : = \left\{\begin{array}{cl} \frac p{1-p} & \text{if} \quad 0<p<1,\\
	\infty &\text{if}\quad p=1, \\
	\frac p{p-1}  &\text{if}\quad 1<p<\infty,\\
	1  &\text{if}\quad p=\infty.
	\end{array}
	\right.
	$$
	
	{\rm (iii)} To state our results we use the notation $p \rw q$ for $0 < p,\,q
	\le \infty$ defined by
	$$
	\frac{1}{p \rw q} = \frac{1}{q} - \frac{1}{p} \qq \mbox{if} \qq q <
	p,
	$$
	and $p \rw q = \infty$ if $q \ge p$.
	
	{\rm (iv)} If $I = (a,b) \subseteq \R$ and $g$ is a monotone
	function on $I$, then by $g(a)$ and $g(b)$ we mean the limits
	$\lim_{t\rw a+}g(t)$ and $\lim_{t\rw b-}g(t)$, respectively.
\end{conv}

Our main results are the following theorems.
Throughout the paper we will denote
$$
\widetilde{V}(x) : = \| v_1^{-1}v_2\|_{p_1 \rw 	p_2,B(0,x)}, \quad \mbox{and} \quad {\mathcal V}(t,x):= \frac{\widetilde{V}(t)}{\widetilde{V}(t)+\widetilde{V}(x)} ~ (t > 0,\,x > 0).
$$
\begin{thm}\label{main01}
	Let $0 < \t_2 = p_2 \leq p_1 = \t_1 < \infty$. Assume that $v_1, v_2\in \W(\rn)$, $\o_1\in \dual{\O_{\t_1}}$ and $\o_2 \in \O_{\t_2}$. Then
	\begin{equation*}
	\|\Id\|_{\dual LM_{p_1\t_1,\o_1}(\rn,v_1)\rw LM_{p_2\t_2,\o_2}(\rn,v_2)}\ap \big\|\|\o_1\|_{p_1,(0,|\cdot|)}^{-1} \,\, \|\o_2\|_{p_2,(|\cdot|,\infty)} \big\|_{p_1\rw p_2,v_1^{-1}v_2,\rn}.
	\end{equation*}
\end{thm}

\begin{thm}\label{main02}
	Let $0 < p_1, \,p_2, \,\t_1, \,\t_2 < \infty$ and $\t_2 \neq p_2 \leq p_1 = \t_1$.  Assume that $v_1, v_2\in \W(\rn)$, $\o_1\in \dual{\O_{\t_1}}$ and $\o_2 \in \O_{\t_2}$.
	
	\rm{(i)} If $p_1\le \t_2$, then
	\begin{equation*}
	\|\Id\|_{\dual LM_{p_1\t_1,\o_1}(\rn,v_1)\rw LM_{p_2\t_2,\o_2}(\rn,v_2)} \ap \sup_{t\in (0,\infty)}  \big\| \|\o_1\|_{p_1,(0,|\cdot|)}^{-1} \big\|_{p_1 \rw  p_2,v_1^{-1}v_2,B(0,t)} \|\o_2\|_{\t_2,(t,\infty)};
	\end{equation*}
	
	\rm{(ii)} If $\t_2 < p_1$, then
	\begin{align*}
	\|\Id\|_{\dual LM_{p_1\t_1,\o_1}(\rn,v_1)\rw LM_{p_2\t_2,\o_2}(\rn,v_2)} 
	\ap \bigg(\int_0^{\infty} \big\| \|\o_1\|_{p_1,(0,|\cdot|)}^{-1} \big\|_{p_1 \rw p_2,v_1^{-1}v_2,B(0,t)}^{p_1\rw \t_2} \,d \bigg(-\|\o_2\|_{\t_2,(t,\infty)}^{p_1\rw \t_2} \bigg) \bigg)^{\frac{1}{p_1\rw \t_2}}.
	\end{align*}
\end{thm}

\begin{thm}\label{main03}
	Let $0<p_1, \,p_2, \,\t_1, \,\t_2 < \infty$ and $\t_2 = p_2 \leq p_1 \neq \t_1$. Assume that $v_1, v_2\in \W(\rn)$, $\o_1\in \dual{\O_{\t_1}}$ and $\o_2 \in \O_{\t_2}$.
	
	{\rm (i)} If $\t_1 \le p_2$, then
	\begin{equation*}
	\|\Id\|_{\dual LM_{p_1\t_1,\o_1}(\rn,v_1)\rw LM_{p_2\t_2,\o_2}(\rn,v_2)}\ap 
	\sup_{t\in (0,\infty)} \|\o_1\|_{\t_1,(0,t)}^{-1} \, \big\| \|\o_2\|_{p_2,(|\cdot|,\infty)} \big\|_{p_1\rw p_2,v_1^{-1}v_2,B(0,t)};
	\end{equation*}
	
	{\rm (ii)} If $p_2 < \t_1$, then
	\begin{align*}
	\|\Id\|_{\dual LM_{p_1\t_1,\o_1}(\rn,v_1)\rw LM_{p_2\t_2,\o_2}(\rn,v_2)}  \ap & \,\, \bigg(\int_0^\infty \big\| \|\o_2\|_{p_2,(|\cdot|,\infty)} \big\|_{p_1\rw p_2,v_1^{-1}v_2,B(0,t)}^{\t_1\rw p_2}  d\,\bigg(-\|\o_1\|_{\t_1,(0,t)}^{-\t_1\rw p_2}\bigg)\bigg)^{\frac{1}{\t_1\rw p_2}} \\
	& + \|\o_1\|_{\t_1,(0,\infty)}^{-1} \big\| \|\o_2\|_{p_2,(|\cdot|,\infty)}\big\|_{p_1\rw p_2,v_1^{-1}v_2,\rn}.
	\end{align*}
\end{thm}

In view of Lemma \ref{LMpp}, Theorems \ref{main01} - \ref{main03} are straightforward consequences of \cite[Theorem~3.1]{mu_emb} and \cite[Theorem~4.2]{mu_emb}.

To state further results we need the following definitions.
\begin{defi}
	Let $U$ be a continuous, strictly increasing function on $[0,\i)$ such that $U(0)=0$ and $\lim_{t\rw\i}U(t)=\i$. Then we say that $U$ is admissible.
\end{defi}

Let $U$ be an admissible function. We say that a function $\vp$ is
$U$-quasiconcave if $\vp$ is equivalent to an increasing function on
$(0,\i)$ and ${\vp} / {U}$ is equivalent to a decreasing function on
$(0,\infty)$. We say that a $U$-quasiconcave function $\vp$ is
non-degenerate if
$$
\lim_{t\rw 0+} \vp(t) = \lim_{t\rw \i} \frac{1}{\vp(t)} = \lim_{t\rw
	\i} \frac{\vp(t)}{U(t)} = \lim_{t\rw 0+} \frac{U(t)}{\vp(t)} =0.
$$
The family of non-degenerate $U$-quasiconcave functions is denoted
by $Q_U$.
\begin{defi}
	Let $U$ be an admissible function, and let $w$ be a non-negative measurable function on $(0,\i)$. We say that the function $\vp$, defined by
	\begin{equation*}
	\vp(t)=U(t)\int_0^{\infty} \frac{w(\tau)\,d\tau}{U(\tau)+U(t)}, \qq t\in (0,\i),
	\end{equation*}
	is a fundamental function of $w$ with respect to $U$. One will also say that $w(\tau)\,d\tau$ is a representation measure of $\vp$ with respect to $U$.
\end{defi}

\begin{rem}\label{nondegrem}
	Let $\vp$ be the fundamental function of $w$ with
	respect to $U$.
	Assume that
	\begin{equation*}
	\int_0^{\infty}\frac{w(\tau)\,d\tau}{U(\tau)+U(t)}<\i, ~ t> 0, \qq \int_0^1 \frac{w(\tau)\,d\tau}{U(\tau)}=\int_1^{\infty}w(\tau)\,d\tau=\infty.
	\end{equation*}
	Then $\vp\in Q_{U}$.
\end{rem}

\begin{rem}\label{limsupcondition}
	Suppose that $\vp (x) < \i$ for all $x \in (0,\i)$, where $\vp$ is
	defined by
	\begin{equation*}
	\vp(x)=\esup_{t\in(0,x)}{U(t)}
	\esup_{\tau\in(t,\i)}\frac{w(\tau)}{U(\tau)},~~t\in\I.
	\end{equation*}
	If
	$$
	\limsup_{t \rightarrow 0 +} w(t) = \limsup_{t \rightarrow +\infty} \frac{1}{w(t)} = \limsup_{t \rightarrow 0 +} \frac{U(t)}{w(t)} = \limsup_{t \rightarrow +\infty} \frac{w(t)}{U(t)} = 0,
	$$
	then
	$\vp\in Q_{U}$.
\end{rem}

\begin{thm}\label{maintheorem1}
	Let $0<p_1, \,p_2, \,\t_1, \,\t_2 < \infty$, $p_2 < p_1$, $\t_1 \leq p_2 < \t_2$. Assume that $v_1, v_2\in \W(\rn)$, $\o_1 \in \dual{\O_{\t_1}}$ and $\o_2\in \O_{\t_2}$. 
	Suppose that $\widetilde{V}$ is admissible and
	\begin{equation*}
	\vp_1(x):= \sup_{t\in (0,\infty)} \widetilde{V}(t) \, {\mathcal V}(x,t) \,\, \|\o_1\|_{\t_1,(0,t)}^{-1} \in Q_{ \widetilde{V}^{\frac{1}{p_1 \rw p_2}}}.
	\end{equation*}
	
	{\rm (i)} If $p_1 \le \t_2$, then
	\begin{equation*}
	\|\Id\|_{\dual LM_{p_1\t_1,\o_1}(\rn,v_1)\rw LM_{p_2\t_2,\o_2}(\rn,v_2)}\ap \sup_{x\in(0,\infty)} \vp_1(x) \sup_{t\in (0,\infty)} {\mathcal V} (t,x) \,\, \|\o_2\|_{\t_2,(t,\infty)}.
	\end{equation*}
	
	{\rm (ii)} If $\t_2 < p_1$, then
	\begin{align*}
	\|\Id\|_{\dual LM_{p_1\t_1,\o_1}(\rn,v_1)\rw LM_{p_2\t_2,\o_2}(\rn,v_2)} \ap \sup_{x\in(0,\infty)} \vp_1(x) \bigg(\int_0^\infty  {\mathcal V}(t,x)^{p_1\rw \t_2}\,d\bigg(-\|\o_2\|_{\t_2,(t,\infty)}^{p_1\rw \t_2}\bigg) \bigg)^{\frac{1}{p_1\rw\t_2}}.
	\end{align*}
\end{thm}

\begin{thm}\label{maintheorem3}
	Let $0 < p_1, \,p_2, \,\t_1, \,\t_2 < \infty$,\ $p_2 < p_1$ and $p_2 < \min\{\t_1,\t_2\}$. Assume that $v_1, v_2\in \W(\rn)$, $\o_1 \in	\dual{\O_{\t_1}}$ and $\o_2\in \O_{\t_2}$.
	Suppose that $\widetilde{V}$ is admissible and
	\begin{equation*}
	\vp_2(x):= \bigg(\int_0^{\infty} [\widetilde{V}(t){\mathcal V}(x,t)]^{\t_1\rw p_2}\,d\bigg(-\|\o_1\|_{\t_1,(0,t)}^{-\t_1\rw p_2}\bigg)\bigg)^{\frac{1}{\t_1 \rw p_2}} \in Q_{{\widetilde{V}}^{\frac{1}{p_1\rw p_2}}}.
	\end{equation*}
	%
	
	{\rm (i)} If $\max\{p_1,\t_1\}\leq \t_2$, then
	\begin{align*}
	\|\Id\|_{\dual LM_{p_1\t_1,\o_1}(\rn,v_1)\rw LM_{p_2\t_2,\o_2}(\rn,v_2)} \ap & \,\, \sup_{x\in (0,\infty)} \vp_2(x) \sup_{t\in (0,\infty)} {\mathcal V}(t,x) \|\o_2\|_{\t_2,(t,\infty)} \\
	& + \|\o_1\|_{\t_1,(0,\infty)}^{-1} \sup_{t\in (0,\infty)} \widetilde{V}(t) \|\o_2\|_{\t_2,(t,\infty)};
	\end{align*}
	
	{\rm (ii)} If $p_1 \leq \t_2 < \t_1$, then
	\begin{align*}
	\|\Id\|_{\dual LM_{p_1\t_1,\o_1}(\rn,v_1)\rw LM_{p_2\t_2,\o_2}(\rn,v_2)} & \\
	& \hspace{-3.5cm} \ap \bigg( \int_0^{\infty} \vp_2(x)^{\frac{\t_1 \rw \t_2 \cdot \t_1 \rw p_2}{\t_2 \rw p_2}} \widetilde{V}(x)^{\t_1 \rw p_2} \bigg(\sup_{t\in(0,\infty)} {\mathcal V}(t,x) \|\o_2\|_{\t_2,(t,\infty)}\bigg)^{\t_1 \rw \t_2} d \bigg( - \|\o_1\|_{\t_1,(0,x)}^{-\t_1 \rw p_2}\bigg) \bigg)^{\frac{1}{\t_1 \rw \t_2}} \\
	& \hspace{-3cm} + \|\o_1\|_{\t_1,(0,\infty)}^{-1} \sup_{t\in (0,\infty)} \widetilde{V}(t) \|\o_2\|_{\t_2,(t,\infty)};
	\end{align*}
	
	{\rm (iii)} If $\t_1 \leq \t_2 < p_1$, then
	\begin{align*}
	\|\Id\|_{\dual LM_{p_1\t_1,\o_1}(\rn,v_1) \rw LM_{p_2\t_2,\o_2}(\rn,v_2)} \ap & \,\, \sup_{x\in (0,\infty)} \vp_2(x) \bigg(\int_0^{\infty} {\mathcal V}(t,x)^{p_1 \rw \t_2} d\bigg(-\|\o_2\|_{\t_2,(t,\infty)}^{p_1 \rw \t_2}\bigg)\bigg)^{\frac{1}{p_1 \rw \t_2}}\\
	&  + \|\o_1\|_{\t_1,(0,\infty)}^{-1} \bigg( \int_0^{\infty} \widetilde{V}(t)^{p_1 \rw \t_2} d\bigg(-\|\o_2\|_{\t_2,(t,\infty)}^{p_1 \rw \t_2}\bigg)\bigg)^{\frac{1}{p_1 \rw \t_2}};
	\end{align*}
	
	{\rm (iv)} If $\t_2 < \min\{p_1,\t_1\}$, then
	\begin{align*}
	\|\Id\|_{\dual LM_{p_1\t_1,\o_1}(\rn,v_1)\rw LM_{p_2\t_2,\o_2}(\rn,v_2)} & \\
	& \hspace{-3.5cm} \ap \bigg( \int_0^{\infty} \vp_2(x)^{\frac{\t_1 \rw \t_2 \cdot \t_1 \rw p_2}{\t_2 \rw p_2}} \widetilde{V}(x)^{\t_1 \rw p_2} \bigg(\int_0^{\infty}{\mathcal V}(t,x)^{p_1 \rw\t_2} d\bigg(-\|\o_2\|_{\t_2,(t,\infty)}^{p_1 \rw \t_2} \bigg)\bigg)^{\frac{\t_1 \rw \t_2}{p_1 \rw \t_2}} d\bigg(-\|\o_1\|_{\t_1,(0,x)}^{-\t_1 \rw p_2} \bigg)\bigg)^{\frac{1}{\t_1 \rw \t_2}} \\
	& \hspace{-3cm} +\|\o_1\|_{\t_1,(0,\infty)}^{-1} \bigg( \int_0^{\infty} \widetilde{V}(t)^{p_1 \rw \t_2} d\bigg(-\|\o_2\|_{\t_2,(t,\infty)}^{p_1 \rw \t_2} \bigg)\bigg)^{\frac{1}{p_1 \rw \t_2}}.
	\end{align*}
\end{thm}

\begin{thm}\label{maintheorem2}
	Let $0 < \t_1 < p = p_1 = p_2 < \t_2 < \infty$. Assume that $v_1, v_2 \in \W(\rn)\cap C(\rn)$, $\o_1 \in \dual{\O_{\t_1}}$ and $\o_2\in \O_{\t_2}$.
	\begin{equation*}
	\|\Id\|_{\dual LM_{p_1\t_1,\o_1}(\rn,v_1)\rw LM_{p_2\t_2,\o_2}(\rn,v_2)}\ap 
	\sup_{t\in \I} \big\|\|\o_1\|_{\t_1,(0,|\cdot|)}^{-1} \big\|_{\infty,v_1^{-1}v_2,B(0,t)} \|\o_2\|_{\t_2,(t,\infty)}.
	\end{equation*}
\end{thm}

\begin{thm}\label{maintheorem4}
	Let $0 < \t_1, \,\t_2 < \infty$ and $0 < p = p_1 = p_2 < \min\{\t_1,\t_2\}$. Assume that $v_1, v_2\in \W(\rn)$ such that $v_1^{-1}v_2 \in C(\rn)$. Suppose that $\o_1 \in \dual{\O_{\t_1}}$, $\o_2\in \O_{\t_2}$ and 
	$$
	0 < \|\o_2^{-1}\|_{\t_2 \rw p, (x,\infty)} < \infty
	$$
	holds for all $x>0$.
	
	{\rm (i)} If $\t_1 \le \t_2$, then
	\begin{align*}
	\|\Id\|_{\dual LM_{p_1\t_1,\o_1}(\rn,v_1)\rw LM_{p_2\t_2,\o_2}(\rn,v_2)} \\ 
	& \hspace{-4cm} \ap \sup_{x\in (0,\infty)} \bigg( \widetilde{V}(x)^{\t_1\rw p} \int_x^{\infty} \, d \bigg(-\|\o_1\|_{\t_1,(0,t)}^{-\t_1\rw p}\bigg)  + \int_0^x  \widetilde{V}(t)^{\t_1\rw p} \, d\bigg(-\|\o_1\|_{\t_1,(0,t)}^{- \t_1\rw p}\bigg) \bigg) ^{{\frac{1}{\t_1\rw p}}}   \|\o_2\|_{\t_2,(x,\infty)} \\
	&\hspace{-3.5cm} +\|\o_1\|_{\t_1,(0,\infty)}^{-1} \sup_{t \in (0,{\infty})} \widetilde{V}(t) \|\o_2\|_{\t_2,(t,\infty)};
	\end{align*}
	
	{\rm (ii)} If $\t_2 < \t_1$, then
	\begin{align*}
	\|\Id\|_{\dual LM_{p_1\t_1,\o_1}(\rn,v_1)\rw LM_{p_2\t_2,\o_2}(\rn,v_2)} & \\
	& \hspace{-4.5cm}\ap \bigg(\int_0^{\infty} \bigg(\int_x^{\infty} d\bigg(- \|\o_1\|_{\t_1,(0,t)}^{-\t_1\rw p}\bigg)\bigg)^{\frac{\t_1\rw \t_2}{\t_2 \rw p}}  \bigg(\sup_{0 < \tau \leq x} \widetilde V(\tau) \|\o_2\|_{\t_2,(\tau,\infty)}\bigg)^{\t_1 \rw \t_2} d\bigg(-\|\o_1\|_{\t_1,(0,x)}^{-\t_1 \rw p}\bigg)\bigg)^{\frac{1}{\t_1 \rw \t_2}} \\
	& \hspace{-4cm} + \bigg(\int_0^{\infty} \bigg(\int_0^x \widetilde V(t)^{\t_1\rw p} d\bigg(-\|\o_1\|_{\t_1,(0,t)}^{-\t_1\rw p} \bigg) \bigg)^{\frac{\t_1\rw \t_2}{\t_2\rw p}} \widetilde V(x)^{\t_1\rw p} \|\o_2\|_{\t_2,(t,\infty)}^{\t_1\rw \t_2} d\bigg(-\|\o_1\|_{\t_1,(0,x)}^{-\t_1\rw p} \bigg)\bigg)^{\frac{1}{\t_1\rw \t_2}}\\
	& \hspace{-4cm} + \|\o_1\|_{\t_1,(0,\infty)}^{-1} \sup_{t\in (0,{\infty})} \widetilde V(t) \|\o_2\|_{\t_2,(t,\infty)}.
	\end{align*}
\end{thm}


\section{Proofs of main results}\label{s7}


Before proceeding to the proof of our main results we recall the following integration in polar coordinates formula. 

We denote the unit sphere $\{x\in \rn : |x|=1\}$ in $\rn$ by $S^{n-1}$. If $x\in \rn \backslash \{0\}$, the polar coordinates of $x$ are
$$
r=|x| \in (0,\infty), \qquad x'=\frac{x}{|x|} \in S^{n-1}.
$$
There is a unique Borel measure $\sigma = \sigma_{n-1}$ on $S^{n-1}$ such that if $f$ is Borel measurable on $\rn$ and $f\geq 0$ or $f\in L^1(\rn)$, then 
$$
\int_{\rn} f(x) \,dx = \int_0^{\infty} \int_{S^{n-1}} f(rx') r^{n-1} d\sigma(x') dr
$$
(see, for instance, \cite[p. 78]{Folland}).

\begin{lem}\label{triviality}
	Let $0 < p_1, \,p_2, \,\t_1, \,\t_2 \leq \infty$ and $p_1 < p_2$. Assume that $v_1, v_2\in \W(\rn)$, $\o_1\in \dual{\O_{\t_1}}$ and $\o_2\in \O_{\t_2}$. Then $\dual{LM_{p_1\t_1,\o_1}(\rn,v_1)}\not \hookrightarrow LM_{p_2\t_2,\o_2}(\rn,v_2)$.
\end{lem}
\begin{proof}
	Assume that $\dual{LM_{p_1\t_1,\o_1}(\rn,v_1)} \hookrightarrow LM_{p_2\t_2,\o_2}(\rn,v_2)$ holds. Then there exist $c>0$ such that
	\begin{equation*}
	\|f\|_{LM_{p_2\t_2,\o_2}(\rn,v_2)} \leq c \ \|f\|_{\dual{LM_{p_1\t_1,\o_1}(\rn,v_1)}}
	\end{equation*}
	holds for all $f\in \M^+(\rn)$. Let $\tau\in (0,\infty)$ and $f\in \M(\rn)$:  $\supp f\subset  B(0,\tau)$. It is easy to see that
	\begin{align}
	\|f\|_{LM_{p_2\t_2,\o_2}(\rn,v_2)}&= \big\| \|f\|_{p_2,v_2,B(0,t)} \big\|_{\t_2,\o_2,(0,\infty)}\notag\\
	&\geq \big\| \|f\|_{p_2,v_2,B(0,t)} \big\|_{\t_2,\o_2,(\tau,\infty)}\notag\\
	&\geq \|\o_2 \|_{\t_2,(\tau,\infty)} \, \|f\|_{p_2,v_2,B(0,\tau)} \label{1}
	\end{align}
	and
	\begin{align}
	\|f\|_{\dual LM_{p_1\t_1,\o_1}(\rn,v_1)}&= \big\| \|f\|_{p_1,v_1,\dual B(0,t)} \big\|_{\t_1,\o_1,(0,\infty)}\notag\\
	&= \big\| \|f\|_{p_1,v_1,\dual B(0,t)} \big\|_{\t_1,\o_1,(0,\tau)}\notag\\
	&\leq \| \o_1 \|_{\t_1,(0,\tau)} \, \|f\|_{p_1,v_1, B(0,\tau)} \label{2}.
	\end{align}
	
	Combining \eqref{1} with \eqref{2}, we can assert that
	\begin{equation*}
	\| \o_2 \|_{\t_2,(\tau,\infty)} \, \|f\|_{p_2,v_2,B(0,\tau)}\leq c \, \| \o_1 \|_{\t_1,(0,\tau)} \,  \|f\|_{p_1,v_1,B(0,\tau)}.
	\end{equation*}
	Since $\o_1\in \dual{\O_{\t_1}}$ and $\o_2\in \O_{\t_2}$, we conclude that $L_{p_1}(B(0,\tau),v_1)\hookrightarrow L_{p_2}(B(0,\tau),v_2)$, which is a contradiction.
\end{proof}

The following lemma is true.
\begin{lem}\label{mainlemma}
	Let $0 < p_1, \,p_2, \,\t_1, \,\t_2 < \infty$, $p_2 \leq p_1$ and $p_2 < \t_2$. Assume that $v_1, v_2\in \W(\rn)$, $\o_1 \in \dual{\O_{\t_1}}$ and $\o_2\in \O_{\t_2}$. Then
	\begin{align*}
	\|\Id\|_{\dual LM_{p_1\t_1,\o_1}(\rn,v_1)\rw LM_{p_2\t_2,\o_2}(\rn,v_2)}
	= \left\{\sup_{g\in \M^+\I} \ddfrac{\|\Id\|_{\dual LM_{p_1\t_1,\o_1}(\rn,v_1)\rw L_{p_2}\big(v_2(\cdot)H^*g(|\cdot|)^{\frac{1}{p_2}}\big)}^{p_2}}{\|g\|_{\frac{\t_2}{\t_2-p_2},\o_2^{-p_2},(0,\infty)}} \right\}^{\frac{1}{p_2}}.
	\end{align*}
\end{lem}
\begin{proof}
	By duality, interchanging suprema, we have that
	\begin{align*}
	\|\Id\|_{\dual LM_{p_1\t_1,\o_1}(\rn,v_1)\rw LM_{p_2\t_2,\o_2}(\rn,v_2)} & \\
	& \hspace{-3cm} =\sup_{f\in \M^+(\rn)} \ddfrac{\|f\|_{LM_{p_2\t_2,\o_2}(\rn,v_2)}}{\|f\|_{\dual LM_{p_1\t_1,\o_1}(\rn,v_1)}} \\
	& \hspace{-3cm} = \sup_{f\in \M^+(\rn)} \ddfrac{1}{\|f\|_{\dual LM_{p_1\t_1,\o_1}(\rn,v_1)}} \sup_{g\in \M^+\I} \ddfrac{\bigg(\int_0^\infty \bigg(\int_{B(0,\tau)} f(x)^{p_2}  v_2(x)^{p_2} \, dx \bigg) g(\tau)\,d\tau\bigg)^{\frac{1}{p_2}}} {\|g\|_{\frac{\t_2}{\t_2-p_2},\o_2^{-p_2},(0,\infty)}^{\frac{1}{p_2}}}\\
	& \hspace{-3cm} = \sup_{g\in \M^+\I} \ddfrac{1}{\|g\|_{\frac{\t_2}{\t_2-p_2},\o_2^{-p_2},(0,\infty)}^{\frac{1}{p_2}}} \sup_{f\in \M^+(\rn)} \ddfrac{\bigg(\int_0^\infty \bigg(\int_{B(0,\tau)} f(x)^{p_2} v_2(x)^{p_2} \,dx\bigg)  g(\tau) \,d\tau\bigg)^{\frac{1}{p_2}}} {\|f\|_{\dual LM_{p_1\t_1,\o_1}(\rn,v_1)}}.
	\end{align*}
	
	Applying Fubini's Theorem, we get that
	\begin{align}
	\|\Id\|_{\dual LM_{p_1\t_1,\o_1}(\rn,v_1)\rw LM_{p_2\t_2,\o_2}(\rn,v_2)} & \notag \\
	& \hspace{-3cm} = \sup_{g\in \M^+\I} \ddfrac{1}{\|g\|_{\frac{\t_2}{\t_2-p_2},\o_2^{-p_2},(0,\infty)}^{\frac{1}{p_2}}} \sup_{f\in \M^+(\rn)} \ddfrac{\bigg(\int_{\rn} f(x)^{p_2} v_2(x)^{p_2} \bigg(\int_{|x|}^{\infty} g(\tau) \,d\tau \bigg)\, dx \bigg)^{\frac{1}{p_2}}} {\|f\|_{\dual LM_{p_1\t_1,\o_1}(\rn,v_1)}}\notag\\
	& \hspace{-3cm} = \sup_{g\in \M^+\I} \ddfrac{1}{\|g\|_{\frac{\t_2}{\t_2-p_2},\o_2^{-p_2},(0,\infty)}^{\frac{1}{p_2}}} \|\Id\|_{\dual LM_{p_1\t_1,\o_1}(\rn,v_1)\rw L_{p_2}\big(v_2(\cdot)H^*g(|\cdot|)^{\frac{1}{p_2}}\big)}. \label{GommeninIlkDenkligi}
	\end{align}
\end{proof}

\noindent{\bf Proof of Theorem \ref{maintheorem1}.}
By Lemma \ref{mainlemma}, we have that
\begin{align*}
\|\Id\|_{\dual LM_{p_1\t_1,\o_1}(\rn,v_1)\rw LM_{p_2\t_2,\o_2}(\rn,v_2)} =
\sup_{g\in \M^+\I}
\frac{1}{\|g\|_{\frac{\t_2}{\t_2-p_2},\o_2^{-p_2},(0,\infty)}^{\frac{1}{p_2}}}
\|\Id\|_{\dual LM_{p_1\t_1,\o_1}(\rn,v_1)\rw L_{p_2}\big(v_2(\cdot)
	H^*g(|\cdot|)^{\frac{1}{p_2}}\big)}.
\end{align*}

Since $\t_1\leq p_2$, applying \cite[Theorem~4.2, (a)]{mu_emb}, we obtain that
\begin{align*}
\|\Id\|_{\dual LM_{p_1\t_1,\o_1}(\rn,v_1)\rw LM_{p_2\t_2,\o_2}(\rn,v_2)} \ap
\left\{ \sup_{g\in \M^+(0,\infty)} \ddfrac {\sup_{t\in (0,\infty)}
	\|\o_1\|_{\t_1,(0,t)}^{- p_2} \|H^*g(|\cdot|)\|_{\frac{p_1}{p_1 - p_2},(v_1^{-1}v_2)^{p_2},B(0,t)}}
{\|g\|_{\frac{\t_2}{\t_2-p_2},\o_2^{-p_2},(0,\infty)}}\right\}^{\frac{1}{p_2}}.
\end{align*}

Using polar coordinates, we have that
\begin{equation*}
\|H^*g(|\cdot|)\|_{\frac{p_1}{p_1-p_2},(v_1^{-1}v_2)^{p_2},B(0,t)}=\|H^*g\|_{\frac{p_1}{p_1-p_2},{\tilde v}^{\frac{p_1-p_2}{p_1}},(0,t)}, \quad t>0,
\end{equation*}
where
\begin{equation*}
\tilde{v}(r):=\int_{S^{n-1}} (v_1^{-1}v_2)(rx')^{\frac{p_1p_2}{p_1- p_2}}r^{n-1}d\sigma(x'), \quad r > 0.
\end{equation*}

Thus, we obtain that\begin{align*}
\|\Id\|_{\dual LM_{p_1\t_1,\o_1}(\rn,v_1)\rw LM_{p_2\t_2,\o_2}(\rn,v_2)} \ap \left\{\sup_{g\in \M^+\I} \ddfrac{\sup_{t\in(0,\infty)} \|\o_1\|_{\t_1,(0,t)}^{-p_2} \left\|H^*g\right\|_{\frac{p_1}{p_1-p_2},{\tilde v}^{\frac{p_1-p_2}{p_1}},(0,t)}}  {\|g\|_{\frac{\t_2}{\t_2-p_2},\o_2^{-p_2},(0,\infty)}}\right\}^{\frac{1}{p_2}}.
\end{align*}

Taking into account that
\begin{align}
\int_0^t \tilde{v} (r)\,dr & = \int_0^t \int_{S^{n-1}} (v_1^{-1}v_2)(rx')^{\frac{p_1p_2}{p_1-p_2}}\,d\sigma(x') r^{n-1} dr \notag \\
& = \int_{B(0,t)} (v_1^{-1}v_2)^{\frac{p_1p_2}{p_1-p_2}} (x)\,dx = \widetilde V(t)^{\frac{p_1p_2}{p_1-p_2}}, \label{Vtilde}
\end{align}

\rm{(i)} if $p_1\le \theta_2$, then applying \cite[Theorem 3.2,
(i)]{gmp}, we arrive at
\begin{equation*}
\|\Id\|_{\dual LM_{p_1\t_1,\o_1}(\rn,v_1)\rw LM_{p_2\t_2,\o_2}(\rn,v_2)} \ap
\sup_{x\in (0,\infty)} \vp_1(x) \sup_{t\in (0,\infty)} {\mathcal V}(t,x)
\|\o_2\|_{\t_2,(t,\infty)};
\end{equation*}

\rm{(ii)} if $\t_2 < p_1$, then applying \cite[Theorem 3.2, (ii)]{gmp}, we arrive at
\begin{align*}
\|\Id\|_{\dual LM_{p_1\t_1,\o_1}(\rn,v_1)\rw LM_{p_2\t_2,\o_2}(\rn,v_2)} \ap
\sup_{x\in (0,\infty)} \vp_1(x) \bigg(\int_0^\i {\mathcal V}(t,x)^{p_1 \rw
	\t_2}  d\bigg( - \|\o_2\|_{\t_2,(t,\infty)}^{p_1 \rw \t_2}\bigg)
\bigg)^{\frac{1}{p_1 \rw \t_2}}.
\end{align*}

The proof is completed.

\hspace{16.9cm}$\square$

\begin{rem}
	In view of Remark~\ref{limsupcondition}, if
	\begin{align*}
	\limsup_{t\rw 0+} \widetilde V(t) \|\o_1\|_{\t_1,(0,t)}^{-1} & = \limsup_{t\rw
		+\infty} \widetilde V(t)\|\o_1\|_{\t_1,(0,t)} \\
	& = \limsup_{t\rw 0+}\|\o_1\|_{\t_1,(0,t)} = \limsup_{t\rw +\infty}
	\|\o_1\|_{\t_1,(0,t)}^{-1} = 0,
	\end{align*}
	then $\vp_1 \in Q_{\widetilde V^{\frac{1}{p_1 \rw p_2}}}$.
\end{rem}

\noindent{\bf Proof of Theorem \ref{maintheorem3}.}
By  Lemma \ref{mainlemma}, applying \cite[Theorem~4.2, (c)]{mu_emb},  we have that
\begin{align*}
\|\Id\|_{\dual LM_{p_1\t_1,\o_1}(\rn,v_1)\rw LM_{p_2\t_2,\o_2}(\rn,v_2)} \ap & \,\, \|\o_1\|_{\t_1,(0,\infty)}^{-1} \left\{\sup_{g\in \M^+\I} \ddfrac{ \|H^*g(|\cdot|)\|_{\frac{p_1}{p_1 - p_2},(v_1^{-1}v_2)^{p_2},\rn}} {\|g\|_{\frac{\t_2}{\t_2 - p_2},\o_2^{-p_2},(0,\infty)}} \right\}^{\frac{1}{p_2}} \\
& + \left\{\sup_{g\in \M^+\I} \ddfrac{\bigg(\int_0^\i \|H^* g(|\cdot|)\|_{\frac{p_1}{p_1 - p_2},(v_1^{-1}v_2)^{p_2},B(0,t)}^{\frac{\t_1}{\t_1 - p_2}} d \bigg( - \|\o_1\|_{\t_1,(0,t)}^{- \frac{\t_1 p_2}{\t_1 - p_2}} \bigg)\bigg)^{\frac{\t_1 - p_2}{\t_1}}}  {\|g\|_{\frac{\t_2}{\t_2 - p_2},\o_2^{-p_2},(0,\infty)}} \right\}^{\frac{1}{p_2}}.
\end{align*}

Using polar coordinates, we have that
\begin{align*}
\|\Id\|_{\dual LM_{p_1\t_1,\o_1}(\rn,v_1)\rw LM_{p_2\t_2,\o_2}(\rn,v_2)} \ap & \,\, \|\o_1\|_{\t_1,(0,\infty)}^{-1} \left\{\sup_{g\in \M^+\I} \ddfrac{ \|H^*g\|_{\frac{p_1}{p_1-p_2},{\tilde v}^{\frac{p_1-p_2}{p_1}},(0,\infty)}} {\|g\|_{\frac{\t_2}{\t_2-p_2},\o_2^{-p_2},(0,\infty)}}\right\}^{\frac{1}{p_2}} \notag \\
& + \left\{\sup_{g\in \M^+\I} \ddfrac{\bigg(\int_0^{\infty} \|H^*g\|_{\frac{p_1}{p_1-p_2},{\tilde v}^{\frac{p_1-p_2}{p_1}},(0,t)}^{\frac{\t_1}{\t_1-p_2}} d\bigg(-\|\o_1\|_{\t_1,(0,t)}^{-\frac{\t_1p_2}{\t_1-p_2}}\bigg)\bigg)^{\frac{\t_1-p_2}{\t_1}}}  {\|g\|_{\frac{\t_2}{\t_2-p_2},\o_2^{-p_2},(0,\infty)}}\right\}^{\frac{1}{p_2}}\\
:= & \,\, C_1 + C_2.
\end{align*}

Assume first that $p_1\leq \t_2$. On using the characterization of the boundedness of the operator $H^*$ in weighted Lebesgue spaces (see, for instance, \cites{ok,kp}), we arrive at
\begin{equation*}
C_1 \ap  \|\o_1\|_{\t_1,(0,\infty)}^{-1} \sup_{t\in (0,\infty)} \widetilde V(t)\,
\|\o_2\|_{\t_2,(t,\infty)}.
\end{equation*}

{\rm (i)} Let  $\t_1\le \t_2$. Applying \cite[Theorem 3.1, (i)]{gmp}, we obtain that
\begin{equation*}
C_2 \ap \sup_{x\in (0,\infty)} \vp_2(x) \, \sup_{t\in (0,\infty)} {\mathcal V} (t,x) \,\|\o_2\|_{\t_2,(t,\infty)}.
\end{equation*}
Consequently, the proof is completed in this case.

{\rm (ii)} Let $\t_2<\t_1$. Using \cite[Theorem 3.1, (ii)]{gmp}, we have that
\begin{align*}
C_2 \ap \bigg( \int_0^\i \vp_2(x)^{\frac{\t_1 \rw \t_2 \cdot \t_1 \rw p_2}{\t_2 \rw p_2}} \widetilde{V}(x)^{\t_1 \rw p_2}  \bigg(\sup_{t\in(0,\infty)}
{\mathcal V}(t,x) \|\o_2\|_{\t_2,(t,\infty)}\bigg)^{\t_1 \rw \t_2}\, d\bigg( - \|\o_1\|_{\t_1,(0,x)}^{- \t_1 \rw p_2}\bigg) \bigg)^{\frac{1}{\t_1 \rw \t_2}},
\end{align*}
and the statement follows in this case.

Let us now assume that $\t_2 < p_1$. Then, using the characterization of the boundedness of the operator $H^*$ in weighted Lebesgue spaces, we have that
\begin{equation*}
C_1 \ap \|\o_1\|_{\t_1,(0,\infty)}^{-1} \bigg( \int_0^\i \widetilde V(t)^{p_1 \rw \t_2} \,d\bigg( - \|\o_2\|_{\t_2, (t,\infty)}^{p_1 \rw \t_2}\bigg) \bigg)^{\frac{1}{p_1 \rw \t_2}}.
\end{equation*}

{\rm (iii)} Let $\t_1\leq \t_2$, then \cite[Theorem 3.1, (iii)]{gmp} yields that
\begin{align*}
C_2\ap \sup_{x\in (0,\infty)} \vp_2(x) \bigg(\int_0^{\infty} {\mathcal V}(t,x)^{p_1 \rw \t_2} d \bigg( - \|\o_2\|_{\t_2, (t,\infty)}^{p_1 \rw
	\t_2}\bigg)\bigg)^{\frac{1}{p_1 \rw \t_2}},
\end{align*}
and these completes the proof in this case.

{\rm (iv)} If $\t_2 < \t_1$, then on using \cite[Theorem 3.1,
(iv)]{gmp}, we arrive at
\begin{align*}
C_2 \ap  \bigg( \int_0^\infty \vp_2(x)^{\frac{\t_1 \rw \t_2 \cdot \t_1 \rw p_2}{\t_2 \rw p_2}} \widetilde{V}(x)^{\t_1 \rw p_2} \bigg(\int_0^{\infty}{\mathcal V}(t,x)^{p_1 \rw \t_2} d \bigg(-\|\o_2\|_{\t_2,(t,\infty)}^{p_1 \rw \t_2} \bigg)
\bigg)^{\frac{\t_1 \rw \t_2}{p_1 \rw \t_2}}
d\bigg(-\|\o_1\|_{\t_1,(0,x)}^{- \t_1 \rw p_2} \bigg)
\bigg)^{\frac{1}{\t_1 \rw \t_2}},
\end{align*}
and in this case the proof is completed. 

\hspace{16.9cm}$\square$

\begin{rem}
	Assume that $\vp_2(x) < \infty, ~ x > 0$. In view of
	Remark~\ref{nondegrem}, if
	$$
	\int_0^1 \bigg(\int_0^t  \o_1^{\t_1} \bigg)^{-\frac{\t_1}{\t_1 - p_2}}
	\o_1^{\t_1}(t) \,dt = \int_1^{\infty} \widetilde V(t)^{\frac{\t_1 p_2}{\t_1 - p_2}}
	\bigg(\int_0^t  \o_1^{\t_1} \bigg)^{-\frac{\t_1}{\t_1-p_2}} \o_1^{\t_1}(t) \,dt = \infty,
	$$
	then $\vp_2 \in Q_{\widetilde V^{\frac{1}{p_1 \rw p_2}}}$.
\end{rem}

\noindent{\bf Proof of Theorem \ref{maintheorem2}.}
By  Lemma \ref{mainlemma}, applying \cite[Theorem~4.2, (b)]{mu_emb},  we get that
\begin{align*}
\|\Id\|_{\dual LM_{p_1\t_1,\o_1}(\rn,v_1)\rw LM_{p_2\t_2,\o_2}(\rn,v_2)} = \left\{ \sup_{g\in \M^+(0,\infty)} \ddfrac{\sup_{t\in (0,\infty)} \|\o_1\|_{\t_1,(0,t)}^{-p} \|H^*g(|\cdot|)\|_{\infty,(v_1^{-1}v_2)^{p},B(0,t)}} {\|g\|_{\frac{\t_2}{\t_2-p},\o_2^{-p},(0,\infty)}} \right\}^{\frac{1}{p}}.
\end{align*}

Recall that, whenever $F,G$ are non-negative measurable functions on $(0,\infty)$ and $F$ is non-increasing, then 
\begin{equation}\label{esupFG1}
\esup_{t\in (0,\infty)} F(t)G(t) = \esup_{t\in (0,\infty)} F(t) \esup_{\tau\in (0,t)} G(\tau).
\end{equation}

Observe that 
\begin{equation}\label{polcorforsup}
\|H^*g(|\cdot|)\|_{\infty,(v_1^{-1}v_2)^p,B(0,t)} = \sup_{\tau\in(0,t)} \, \sup_{|y|=\tau} \big( v_1^{-1}(y) v_2(y)\big)^p H^* g (|y|)  = \|H^*g\|_{\infty,\tilde{\tilde {v}}, (0,t)}
\end{equation}
holds for all $t > 0$, where $\tilde{\tilde {v}}(\tau):= \big(\sup_{|y|=\tau} v_1^{-1}(y)v_2(y)\big)^p$, $\tau > 0$. 

On using \eqref{esupFG1}, we get that
\begin{align*}
\|\Id\|_{\dual LM_{p_1\t_1,\o_1}(\rn,v_1)\rw LM_{p_2\t_2,\o_2}(\rn,v_2)} & = \left\{ \sup_{g\in \M^+(0,\infty)} \ddfrac{\sup_{t\in (0,\infty)} \|\o_1\|_{\t_1,(0,t)}^{-p} \|H^*g\|_{\infty,\tilde{\tilde {v}}, (0,t)}} {\|g\|_{\frac{\t_2}{\t_2-p},\o_2^{-p},(0,\infty)}} \right\}^{\frac{1}{p}} \\
& = \left\{ \sup_{g\in \M^+(0,\infty)} \ddfrac{ \|H^*g\|_{\infty,\|\o_1\|_{\t_1,(0,\cdot)}^{-p} \tilde{\tilde{v}}(\cdot),(0,\infty)}
} {\|g\|_{\frac{\t_2}{\t_2-p},\o_2^{-p},(0,\infty)}} \right\}^{\frac{1}{p}}.
\end{align*}

Using the characterization of the boundedness of $H^*$ in weighted Lebesgue spaces, we obtain that
\begin{align*}
\|\Id\|_{\dual LM_{p_1\t_1,\o_1}(\rn,v_1)\rw LM_{p_2\t_2,\o_2}(\rn,v_2)} & \ap \sup_{t\in (0,\infty)} \|\o_2\|_{\t_2,(t,\infty)} \bigg(\sup_{s\in(0,t)} \|\o_1\|_{\t_1,(0,s)}^{-1} \tilde{\tilde v}(s)^{\frac{1}{p}}\bigg) \\
& = \sup_{t\in (0,\infty)} \|\o_2\|_{\t_2,(t,\infty)}
\bigg(\sup_{s\in (0,t)} \sup_{|y|=s} \|\o_1\|_{\t_1,(0,|y|)}^{-1} v_1^{-1}(y) v_2(y)\bigg)\\
& = \sup_{t\in (0,\infty)} \|\o_2\|_{\t_2,(t,\infty)}
\bigg(\sup_{x\in B(0,t)} \|\o_1\|_{\t_1,(0,|x|)}^{-1} v_1^{-1}(x) v_2(x)\bigg)\\  
& = \sup_{t\in (0,\infty)} \|\o_2\|_{\t_2,(t,\infty)}
\bigg\| \|\o_1\|_{\t_1,(0,|\cdot|)}^{-1} \bigg\|_{\infty,v_1^{-1} v_2,B(0,t)}.
\end{align*}

\hspace{16.9cm}$\square$

\noindent{\bf Proof of Theorem \ref{maintheorem4}.}
By  Lemma \ref{mainlemma}, applying \cite[Theorem~4.2, (d)]{mu_emb}, and using \eqref{polcorforsup}, we get that
\begin{align*}
\|\Id\|_{\dual LM_{p_1\t_1,\o_1}(\rn,v_1)\rw LM_{p_2\t_2,\o_2}(\rn,v_2)} \ap & \,\, \|\o_1\|_{\t_1,(0,\infty)}^{-1} \left\{\sup_{g\in \M^+(0,\infty)} \ddfrac{\|H^*g\|_{\infty,\tilde{\tilde{v}},(0,\infty)}} {\|g\|_{\frac{\t_2}{\t_2-p},\o_2^{-p},(0,\infty)}} \right\}^{\frac{1}{p}}\\
&  + \left\{\sup_{g\in \M^+(0,\infty)} \ddfrac{\bigg(\int_0^{\infty} \|H^*g\|_{\infty,\tilde{\tilde {v}}, (0,t)} ^{\frac{\t_1}{\t_1 - p}} d\bigg(-\|\o_1\|_{\t_1,(0,t)}^{-\frac{\t_1 p}{\t_1 - p}}\bigg)\bigg)^{\frac{\t_1 - p}{\t_1}}}  {\|g\|_{\frac{\t_2}{\t_2 - p},\o_2^{-p},(0,\infty)}}\right\}^{\frac{1}{p}}\\
:= & \,\, C_3 + C_4.
\end{align*}

Again, using the characterization of the boundedness of $H^*$ in weighted Lebesgue spaces, we obtain that
\begin{equation*}
C_3\ap \|\o_1\|_{\t_1,(0,\infty)}^{-1} \sup_{t\in (0,{\infty})} \widetilde V(t) \|\o_2\|_{\t_2,(t,\infty)}.
\end{equation*}

{\rm (i)} Let  $\t_1 \le \t_2$, then by \cite[Theorem 4.1]{gop}, we have that
\begin{align*}
C_4  \ap \sup_{x\in (0,\infty)} \bigg( \widetilde{V}(x)^{\t_1\rw p} \int_x^{\infty} \, d \bigg(-\|\o_1\|_{\t_1,(0,t)}^{-\t_1\rw p}\bigg) + \int_0^x  \widetilde{V}(t)^{\t_1\rw p} \, d\bigg(-\|\o_1\|_{\t_1,(0,t)}^{- \t_1\rw p}\bigg) \bigg)^{{\frac{1}{\t_1\rw p}}} \|\o_2\|_{\t_2,(x,\infty)},
\end{align*}
and the statement follows in this case.

{\rm (ii)} Let  $\t_2 < \t_1$, then \cite[Theorem 4.4]{gop} yields that
\begin{align*}
C_4  \ap & \,\, \bigg(\int_0^{\infty} \bigg(\int_x^{\infty} d\bigg(- \|\o_1\|_{\t_1,(0,t)}^{-\t_1\rw p}\bigg)\bigg)^{\frac{\t_1\rw \t_2}{\t_2 \rw p}} \bigg(\sup_{0 < \tau \leq x} \widetilde V(\tau) \|\o_2\|_{\t_2,(\tau,\infty)}\bigg)^{\t_1 \rw \t_2} d\bigg(-\|\o_1\|_{\t_1,(0,x)}^{-\t_1 \rw p}\bigg)\bigg)^{\frac{1}{\t_1 \rw \t_2}} \\
& + \bigg(\int_0^{\infty} \bigg(\int_0^x \widetilde V(t)^{\t_1\rw p} d\bigg(-\|\o_1\|_{\t_1,(0,t)}^{-\t_1\rw p} \bigg) \bigg)^{\frac{\t_1\rw \t_2}{\t_2\rw p}} \widetilde{V}(x)^{\t_1\rw p} \|\o_2\|_{\t_2,(t,\infty)}^{\t_1\rw \t_2} d\bigg(-\|\o_1\|_{\t_1,(0,x)}^{-\t_1\rw p} \bigg)\bigg)^{\frac{1}{\t_1\rw \t_2}},
\end{align*}
and the proof is completed in this case.

\hspace{16.9cm}$\square$

\

{\bf Acknowledgments.} The research of A. Gogatishvili was partially supported by the grant P201-13-14743S of the Grant Agency of the Czech Republic and RVO: 67985840 and  by Shota Rustaveli National Science Foundation grants no. DI/9/5-100/13 (Function spaces, weighted inequalities for integral operators and problems of summability of Fourier series).


\end{document}